\documentclass[12pt]{article}
\usepackage{mathtools}
\usepackage{amsmath}
\usepackage[margin=1in]{geometry}
\usepackage{amsthm}
\usepackage{tikz}
\usepackage{amssymb}
\usepackage{tikz-qtree}
\usepackage{hyperref}
\usetikzlibrary{positioning}
\usepackage{lipsum}
\newtheorem{prop}{Proposition}[section]

\newtheorem{thm}[prop]{Theorem}

\newtheorem{cor}[prop]{Corollary}
\theoremstyle{definition}

\def \ds {\displaystyle}        \def \a {\alpha}         
\begin{document}
\centerline
{
 \bf 
A Generalization of the ``Raboter" operation.
}

\bigskip
\centerline
{\it By Yonah BIERS-ARIEL}
\bigskip

\section{Introduction} In a recent talk at Rutgers' Experimental Math Seminar, Neil Sloane described the ``raboter" operation for the base two representation of a number \cite{talk}. From this representation, one reduces by one the length of each run of consecutive 1s and 0s. Denote this operation by $r(n)$; so, for example, $r(12) = 2$ because 12 is represented in binary as 1100, and reducing the length of each run by one yields 10. 

Sloane also defined $L(k)=\sum_{n=2^k}^{2^{k+1}-1}r(n)$ and conjectured that $L(k) = 2\cdot 3^{k-1} - 2^{k-1}$, a fact which was quickly proven by Doron Zeilberger \cite{drz} and Chai Wah Wu \cite{oeis}.

In Section \ref{bases}, we generalize this theorem to bases other than 2. Let $r(b,n)$ be the the number whose base-$b$ representation is generated by taking the base-$b$ representation of $n$ and shortening each run of consecutive identical elements by one. Further, let $L(b,n) = \sum_{n=b^k}^{b^{k+1}-1} r(b,n)$. We will prove that 
\[L(b,k) = \frac{b(b-1)}{2b-1} (2b-1)^k - \frac{b-1}{2} b^k.\]

In Section \ref{powers}, we raise $r(b,n)$ to various powers. Define $L(p,b,k) = \sum_{n=b^k}^{b^{k+1}-1} r(b,n)^p$; we develop an algorithm in Maple to rigorously compute $L(p,b,k)$ as an expression in terms of $k$ for any fixed $p,b$. In addition, for any fixed $p$, we can conjecture an expression for $L(p,b,k)$ in terms of $b$ and $k$.

\section{More General Bases}\label{bases}
Following the example of Zeilberger, we find a recurrence satisfied by $L(b,k)$ and then find a closed form expression satisfying the same recurrence.
\begin{thm} $\ds L(b,k) = (2b-1)\cdot L(b,k-1) + b^{k-1}\frac{(b-1)^2}{2} $ for $k \ge 2$.
\end{thm}
\begin{proof} 
There are $b^{k+1} - b^{k} = b^{k}(b-1)$ numbers which contribute to $L(b,k)$ are exactly those numbers whose base-$b$ representations use $k+1$ digits, so each can be written as $Ab_1b_2$ where $A \in \{1,\dots,b-1\} \times \{0,\dots,b-1\}^{k-2}$ and $b_1, b_2 \in \{0,\dots,b-1\}$. If $b_1 \neq b_2$, then $b_2$ is a run of just one element, so the raboter operation eliminates it and $r(b,Ab_1b_2) = r(b,Ab_1)$. Numbers with representations $Ab_1$ are exactly those which were counted in the calculation of $L(b,k-1)$, and each is counted $b-1$ times here, once for each $b_2 \neq b_1$. 

If $b_2 = b_1$, then the base-$b$ representation of $r(b,Ab_1b_2)$ is the representation of $r(b,Ab_1)$ with $b_2$ appended to the end, and so $r(b,Ab_1b_2) = b\cdot r(b,Ab_1) + b_2$. Thus,
\begin{align*}
L(b,k) &= \sum_A \sum_{b_1} r(b,Ab_1b_1) + \sum_{b_2 \neq b_1} r(b,Ab_1b_2)
\\
&= \sum_A \sum_{b_1} b\cdot r(b,Ab_1) +b_1 + \sum_{b_2 \neq b_1} r(b,Ab_1)
\\
&= \sum_A \sum_{b_1} (2b-1)r(b,Ab_1) + \sum_A \sum_{b_1} b_1
\\
&= (2b-1)L(b,k-1) + (b-1)b^{k-2}\frac{b(b-1)}{2}
\\
&= (2b-1)L(b,k-1) + b^{k-1}\frac{(b-1)^2}{2}.
\end{align*}
\end{proof}
Together with initial condition $L(b,1) = \frac{b(b-1)}{2}$, this determines the sequence $(L(b,k))_{k=1}^\infty$. Finding an explicit formula for $L(b,k)$ is now just a matter of finding a formula which obeys this same recurrence. 
\begin{cor}$ L(b,k) = \frac{b(b-1)}{2b-1} (2b-1)^k - \frac{b-1}{2} b^k$.
\end{cor}
\begin{proof}
With some help from Doron Zeilberger's Maple package Cfinite, we conjecture that the formula for $L(b,k)$ has the form $\a_1 (2b-1)^k + \a_2 b^k$, so we solve the system of equations 
\begin{align*}
&\a_1(2b-1) + \a_2 b =  \frac{b(b-1)}{2}
\\
&\a_1(2b-1)^2 +\a_2 b^2 = (2b-1) \frac{b(b-1)}{2} +b\frac{(b-1)^2}{2}
\end{align*}
for $\a_1,\a_2$ and find $\a_1 = \frac{b(b-1)}{2b-1}$ and $\a_2 = -\frac{b}{2} + \frac{1}{2}$.
Let $L'(b,k) = \ds \frac{b(b-1)}{2b-1} (2b-1)^k - \frac{b-1}{2} b^k$. Proving that $L(b,k) = L'(b,k)$ is simply a matter of verifying that $L'(b,1) =  \frac{b(b-1)}{2}$ and $\ds L'(b,k) =(2b-1)\cdot L'(b,k-1) + b^{k-1}\frac{(b-1)^2}{2} $ for $k \ge 2$, which can easily be done using Maple or any other computer algebra system.
\end{proof}

\section{Higher Moments}\label{powers}
With a formula for $L(b,k)$ found, we consider the following additional generalization:
\[L(p,b,k) = \sum_{n=2^k}^{2^{k+1}-1} r(b,n)^p;\]
that is the sum of $r(b,n)^p$ taken over all numbers $n$ whose base-$b$ representation has $k+1$-digits. 
The trick in this case is to work inductively beginning with the (solved) $p=1$ case, and, along the way compute $L(l,p,b,k)$ which we define to be the sum of $r(b,n)^p$ taken over all numbers $n$ whose base-$b$ representation has $k+1$-digits, the last of which is $l$.

In order to compute $L(l,p,b,k)$, we use the following recurrence:
\begin{thm}
$L(l,p,b,k) = (b^p-1)\cdot L(l,p,b,k-1)+ L(p,b,k-1) + \sum_{i=1}^p l^i b^{p-i} \binom{p}{i} L(l,p-i,b,k-1).$
\end{thm}
\begin{proof} The numbers with length-$(k+1)$ base-$b$ representations ending in $l$ are exactly those which can be written as $Ab_1b_2$ with $A\in \{1,\dots,b-1\} \times \{0,\dots,b-1\}^{k-2}, b_1 \in \{0,\dots,b-1\},$ and $b_2 = l$. Therefore,
\begin{align*} L(l,p,b,k) &= \sum_A \Big(\sum_{b_1 \neq l} r(b,Ab_1 l)^p + r(b,All)^p\Big)\\
&= \sum_A \Big(\sum_{b_1 \neq l} r(b,Ab_1)^p + (b\cdot r(b,Al)+l)^p\Big)\\
&= L(p,b,k-1) - L(l,p,b,k-1) + \sum_A \sum_{i=0}^p \binom{p}{i} b^{p-i} r(b,Al)^{p-i} l^{i} \\
&= L(p,b,k-1) - L(l,p,b,k-1) + b^pL(l,p,b,k-1) \\
&\hspace{1cm}+ \sum_{i=1}^p \binom{p}{i} b^{p-i} L(l,p-i,b,k-1)^{p-i} l^{i}\\
&= (b^p-1)\cdot L(l,p,b,k-1)+ L(p,b,k-1) + \sum_{i=1}^p l^i b^{p-i} \binom{p}{i} L(l,p-i,b,k-1).
\end{align*}
\end{proof}
We find a similar recurrence for $L(p,b,k)$.
\begin{thm}
$L(p,b,k) = (b^p+b-1)L(p,b,k-1) + \sum_{l=0}^{b-1} \sum_{i=1}^p b^{p-i} l^i \binom{p}{i}L(l,p-i,b,k-1)$.
\end{thm}
\begin{proof}
Again, note that the numbers counted by $L(p,b,k)$ are those which can be written as $Ab_1b_2$ with $A\in \{1,\dots,b-1\} \times \{0,\dots,b-1\}^{k-2},$ and $b_1,b_2 \in \{0,\dots,b-1\}.$ Therefore, the following equations hold:
\begin{align*} 
L(p,b,k) &= \sum_A \Big(\sum_{b_1 \neq b_2} r(b,Ab_1b_2)^p + \sum_{b_1} r(b,Ab_1b_1)^p\Big) \\
&=\sum_A (b-1)\sum_{b_1} r(b,Ab_1)^p + \sum_A\sum_{b_1} (br(Ab_1)+b_1)^p\\
&=(b-1)L(p,b,k-1) + \sum_A\sum_{b_1}\sum_{i=0}^p \binom{p}{i} b^{p-i}r(Ab_1)^{p-i} b_1^i \\
&= (b-1)L(p,b,k-1) + \sum_A\sum_{b_1} b^{p}r(Ab_1)^{p} + \sum_{b_1}\sum_{i=1}^p \sum_A \binom{p}{i} b^{p-i}r(Ab_1)^{p-i} b_1^i\\
&=(b^p + b-1)L(p,b,k-1) + \sum_{b_1}\sum_{i=1}^p b_1^i b^{p-i}\binom{p}{i} L(b_1,p-i,b,k-1).
\end{align*}
Change the name of $b_1$ to $l$ to maintain consistent notation, and we have derived the claimed equation.
\end{proof}

\section{Maple Implementation}
The Maple package \texttt{raboter.txt} available at \texttt{http://sites.math.rutgers.edu/\~{}yb165/raboter.txt} contains functions to implement this recurrence. The most important are \texttt{SumPowers(b,k,p)} which finds an expression in terms of $k$ for $L(p,b,k)$ (for fixed $b$ and $p$) and \texttt{GuessGeneralForm(b,n,p)} which conjectures an expression in terms of $k$ and $b$ for $L(p,b,k)$ (for fixed $p$).

For example, this package proves that
\[L(2,2,k) = \frac{2}{3}5^k - \frac{1}{6}2^k-\frac{2}{3}3^k\]
and conjectures that
\begin{align*}L(2,b,k) &= \Big(\frac{1}{6}b^2-\frac{1}{6}b-\frac{1}{3}\Big)(b-1)^k+\Big(-\frac{1}{6}b^2+\frac{1}{3}b -\frac{1}{6}\Big) b^k 
\\
&- \frac{b(b-1)}{2b-1} (2b-1)^k + \frac{2b^3+3b^2-3b-2}{6(b^2+b-1)} (b^2+b-1)^k.
\end{align*}

\end{document}